\documentclass{amsart}

 \usepackage{amscd,amsfonts,amssymb,amsmath}
 \usepackage{graphicx,pst-all}
 \usepackage{caption}
 \usepackage{float}
 \usepackage{color}

\usepackage[active]{srcltx}
\usepackage[all]{xy}
\usepackage{subfig}
\usepackage{hyperref}
\usepackage{mathrsfs}

\usepackage{amsmath}
\usepackage{amssymb,latexsym}
\usepackage{mathrsfs}
\usepackage{graphics}
\usepackage{latexsym}
\usepackage{psfrag}
\usepackage{import}
\usepackage{verbatim}
\usepackage{graphicx}
\usepackage{pifont,marvosym}

\newcounter{assu}

\theoremstyle{plain}
\newtheorem{lemma}{Lemma}[section]
\newtheorem{theorem}[lemma]{Theorem}
\newtheorem{proposition}[lemma]{Proposition}
\newtheorem{corollary}[lemma]{Corollary}

\theoremstyle{definition}
\newtheorem{assumption}[assu]{Assumption}

\newtheorem{remark}[lemma]{Remark}

\numberwithin{equation}{section}

\newcommand{\R}{\mathbb{R}}
\newcommand{\N}{\mathbb{N}}

\newcommand{\supp}{\text{\rm supp}}

\newcommand{\gr}{\textrm{graph}}

\newcommand{\ve}{\varepsilon}

\newcommand{\erre}{\mathbb{R}}
\newcommand{\enne}{\mathbb{N}}

\newcommand{\f}{\varphi}

\newcommand{\G}{\mathcal{G}}

\begin{document}
\title{Existence and uniqueness of optimal transport maps}

\author{Fabio Cavalletti and Martin Huesmann}
\thanks{MH gratefully acknowledges funding through CRC 1060.}
\address{RWTH, Department of Mathematics, Templergraben  64, D-52062 Aachen (Germany)}
\email{cavalletti@instmath.rwth-aachen.de}

\address{Universit\"at Bonn, Institut f\"ur angewandte Mathematik, Endenicher Allee 60, D-53115 Bonn (Germany)}
\email{huesmann@iam.uni-bonn.de}


\bibliographystyle{plain}

\begin{abstract}
Let $(X,d,m)$ be a proper, non-branching, metric measure space. We show existence and uniqueness of optimal transport maps for 
cost written as non-decreasing and strictly convex functions of the distance, provided $(X,d,m)$ satisfies 
a new weak property concerning the behavior of $m$ under the shrinking of sets to points, see Assumption \ref{A:mcc}. This in particular covers spaces satisfying the measure contraction property.
\\
We also prove a stability property for Assumption \ref{A:mcc}: If $(X,d,m)$ satisfies Assumption \ref{A:mcc} and $\tilde m = g\cdot m$, for some continuous function $g >0$, then also $(X,d,\tilde m)$ verifies Assumption \ref{A:mcc}. Since these changes in the reference measures do not preserve any Ricci type curvature bounds,
this shows that our condition is strictly weaker than measure contraction property. 
\end{abstract}

\maketitle

\section{Introduction}
In \cite{monge1781memoire}, Gaspard Monge studied the by now famous minimization problem 
\begin{equation}\label{e:monge problem}
\inf_{T_{\sharp}\mu_{0}=\mu_{1}} \int  d(x,T(x))  \mu_{0}(dx), 
\end{equation}
on Euclidean space, where $\mu_0$ and $\mu_1$ are two given probability measures and the minimum is taken over all maps pushing $\mu_0$ forward to $\mu_1$. This problem turned out to be very difficult because the functional is non-linear and the constraint set maybe empty. 70 years ago, Kantorovich \cite{kantorovich2006translocation} came up with a relaxation of the minimization problem \eqref{e:monge problem}. He allowed arbitrary couplings $q$ of the two measures $\mu_0$ and $\mu_1$, which we denote by the set $\Pi(\mu_0,\mu_1),$ and also more general cost functions $c:X\times X\to\R$:
\begin{equation}\label{e:kantorovich problem}
\inf_{q\in \Pi(\mu_0,\mu_1)} \int  c(x,y)  q(dx,dy).
\end{equation}
Minimizers of \eqref{e:kantorovich problem} are called optimal couplings and, therefore, this family of problems is commonly called optimal transport problems. A natural and interesting question is when do these two minimization problems coincide, i.e. when is the or an optimal coupling given by a transportation map. In \cite{Brenier}, Brenier showed using ideas from fluid dynamics that on Euclidean space with cost function $c(x,y)=|x-y|^2$ there is always a unique optimal transportation map as soon as $\mu_0$ is absolutely continuous with respect to the Lebesgue measure. Soon after, McCann \cite{mccann01polar} generalized this result to Riemannian manifolds with more general cost functions including convex functions of the distance.  By now, this result is shown in a wide class of settings, for instance for non-decreasing strictly convex functions of the distance in Alexandrov spaces \cite{bertrand2008existence}, for squared distance on the Heisenberg group \cite{ambrosio:heisenberg}, and recently for the 
squared distance on $\mathsf{CD}(K,N)$ and 
$\mathsf{CD}(K,\infty)$ spaces by Gigli \cite{gigli:maps} and for squared distance cost by Rajala and Ambrosio in a metric Riemannian like framework  \cite{ambrosio2011slopes}. 
\medskip\\
In this paper we show existence and uniqueness of optimal transport maps on proper, non-branching, metric measure spaces 
satisfying a new condition, Assumption \ref{A:mcc}, for cost functions of the form $c(x,y)=h(d(x,y))$, with $h$ strictly convex and non-decreasing. 

Assumption \ref{A:mcc} does not imply any lower curvature bounds in the sense of Lott, Sturm and Villani. 
In particular in Section \ref{S:notRicci} we prove that Assumption \ref{A:mcc} cannot imply the measure contraction property, $\mathsf{MCP}$.
On the other hand the measure contraction property implies Assumption \ref{A:mcc}.
Therefore our result applies to spaces enjoying $\mathsf{MCP}$, 
recovers most of the previously mentioned results and in many cases also extends them. 

To our knowledge this is the first existence result of optimal maps in metric spaces for 
$c(x,y)=h(d(x,y))$, with $h$ strictly convex and non-decreasing with no assumption 
on a lower bound on the Ricci curvature of the space. 
For $h = id$, existence of optimal maps, again with no assumption on the curvature of the metric space, 
has been obtained in \cite{biacava:streconv}.
\medskip\\
The crucial idea for the proof of the main result is to \emph{approximate the $c$-cyclically monotone set} on which the optimal measure is concentrated by means of a  suitably chosen sequence of $c$-cyclically monotone sets representing transports into a discrete target.

We conclude this Introduction by describing the structure of the paper. In Section \ref{S:notation} we introduce the general setting of the paper, define Assumption \ref{A:mcc} and state 
the two main results: the existence of optimal transport maps (Theorem \ref{T:statemain1}) and the stability under changes 
in the reference measure of $(X,d,m)$ of Assumption \ref{A:mcc} (Theorem \ref{T:statemain2}).
In Section \ref{S:notRicci} we prove Theorem \ref{T:statemain2} while Section \ref{S:evolution} and Section \ref{S:final} are devoted to the proof of Theorem \ref{T:statemain1}.

\section{Notation and main result}\label{S:notation}
We now introduce the setting of this article. If not explicitly stated otherwise we will always assume to work in this framework.

Let $(X,d,m)$ be a proper, non-branching, metric measure space, that is 
\begin{itemize}
\item[-] $(X,d)$ is a proper, complete and separable metric space with a non-branching geodesic structure; 
\item[-] $m$ is a positive Borel measure, finite over compact sets whose support coincides with $X$.
\end{itemize}
In case we drop the proper assumption, we will refer to $(X,d,m)$ just as non-branching metric measure space.
Let $\mu_{0}, \mu_{1}$  be probability measures over $X$ and let $h : [0,\infty) \to [0, \infty)$ be a strictly convex and non-decreasing map. 

We study the following minimization problem
\begin{equation}\label{E:Mongeproblem}
\min_{T_{\sharp}\mu_{0}=\mu_{1}} \int h( d(x,T(x) ))  \mu_{0}(dx),
\end{equation}
where $T_\sharp\mu_0$ denotes the push forward of $\mu_0$ under the map $T$. In the sequel, we will often denote the cost function  $h \circ d$ just with $c.$
To get hands on the minimization problem \eqref{E:Mongeproblem} we also study its relaxed form, the Kantorovich problem. Let $\Pi(\mu_{0},\mu_{1})$ be the set of transference plans, i.e.
\[
\Pi(\mu_{0},\mu_{1}) : = \{ \pi \in \mathcal{P}(X \times X) : (P_{1})_{\sharp}\pi = \mu_{0}, (P_{2})_{\sharp}\pi = \mu_{1}\},
\]
where $P_{i} : X \times X \to X$ is the projection map onto the $i$-th component, $P_{i}(x_{1},x_{2}) = x_{i}$ for $i =1,2$.

We will \emph{always assume that $\mu_{0}$ and $\mu_{1}$ have finite $c$-transport distance} in the sense that
\[
\inf \left\{  \int_{X \times X} h(d(x,y)) \pi(dxdy) : \pi \in \Pi(\mu_{0},\mu_{1})  \right\} < \infty.
\]
Recall that a transference plan $\pi \in \Pi(\mu_{0},\mu_{1})$  is said to be $c$-cyclically monotone if there exists $\Gamma$ so that
$\pi(\Gamma) =1$ and for every $N \in \enne$ and every $(x_{1},y_{1}) \dots, (x_{N},y_{N}) \in \Gamma$ it holds 
\[
\sum_{i =1}^{N} c(x_{i},y_{i}) \leq \sum_{i =1}^{N} c(x_{i+1},y_{i}),
\]
with $x_{N+1} = x_{1}$.

We also introduce a few objects connecting geodesics of the space $X$ to optimal transport plans.
Let 
\[
\G(X) \subset C([0,1];X),
\] 
be the set of geodesics endowed with the uniform topology inherited from $C([0,1];X)$. Being a closed subset of $C([0,1];X)$, it is Polish.
For any $t \in [0,1]$ consider the map 
\[
e_{t} : \G(X) \to X, 
\]
the evaluation at time $t$ defined by $e_{t}(\gamma) =\gamma_{t}$. 
For a subset $A\subset X$ and a point $x\in X$ the $t-$intermediate points between $A$ and $x$ are defined as
\begin{equation}\label{E:evodef}
A_{t,x} := e_t(\{\gamma \in\G(X): \gamma_0\in A, \gamma_1=x\})\,.
\end{equation}
Assuming $A$ compact, in a general non-branching metric measure space, the set $A_{t,x}$ is closed. If we also assume the space to be proper, as we do here, 
the set $A_{t,x}$, being bounded, is indeed compact.

This evolution defined as \eqref{E:evodef} will play a fundamental role in our analysis. In particular we make the following
\begin{assumption}\label{A:mcc}
A non-branching, metric measure space $(X,d,m)$ verifies Assumption \ref{A:mcc} if
for every compact set $K\subset X$ 
there exists a measurable function $f:[0,1]\to(0,1]$ with 
\[
\limsup_{t\to 0}f(t) > \frac{1}{2},
\]
and a positive $\delta \leq 1$ such that 
\[
m(A_{t,x})\geq f(t)\cdot m(A), \qquad \forall   0 \leq t \leq \delta,
\]
for any compact set $A \subset K$ and any base point $x \in K$.
\end{assumption}

\bigskip

We can now state the main result of this paper.

\begin{theorem}\label{T:statemain1}
Let $(X,d,m)$ be a proper, non-branching, metric measure space verifying Assumption \ref{A:mcc}. Let $\mu_0$ and $\mu_1$ be two probability measures over $X$ with finite $c$-transport distance.
If $\mu_0\ll m$ and $h$ is strictly-convex and non-decreasing, the optimal transport problem associated to \eqref{E:Mongeproblem}
has a unique solution induced by a map.
\end{theorem}

In detail we will prove that if $\mu_0\ll m$ then any  $c$-cyclically monotone plan $\pi$ is induced by a map $T : X \to X$.
With $\pi$ induced by a map we mean that $\pi=(id,T)_\sharp \mu_0$. This implies that the two minimization problems \eqref{e:monge problem} and \eqref{e:kantorovich problem} coincide. Then a direct Corollary of this result is the uniqueness of the optimal coupling.
We will prove the claim by showing that branching at starting points does not happen almost surely. 

\medskip

Regarding Assumption \ref{A:mcc}, we will prove the following result, that can be understood as a stability property. 
Here the space is not needed to be proper.
\begin{theorem}\label{T:statemain2}
Let $(X,d,m)$ be a non-branching metric measure space verifying Assumption \ref{A:mcc}. Consider a continuous function $g: X \to (0,\infty)$ and the measure $\tilde m : = g\cdot m$. 
Then $(X,d,\tilde m)$ is a non-branching metric measure space verifying Assumption \ref{A:mcc}.
\end{theorem}

\section{On Assumption \ref{A:mcc}}\label{S:notRicci}
It is clear that spaces satisfying the \emph{measure contraction property} -- for a definition we 
refer to \cite{ohta:mcp, sturm:MGH2} -- also satisfy Assumption \ref{A:mcc}. 
However, as we will prove in this section, Assumption \ref{A:mcc} does not imply the measure contraction property or, more in general,
any synthetic Ricci curvature bounds.  

In detail, we will show that if $(X,d,m)$ is a non-branching metric measure space verifying Assumption \ref{A:mcc} and $\tilde m = g m$, 
with $g$ continuous and strictly positive, then also $(X,d,\tilde m)$ verifies Assumption \ref{A:mcc}.
Since this kind of changes in the measure destroy Ricci lower bounds, Assumption \ref{A:mcc} cannot imply any of them. See \cite{sturm:MGH2}, Theorem 1.7.

The setting of this subsection is slightly different from the remaining of this note, so we will specify all the assumptions needed in each statement.
We start with two simple lemmas.

\begin{lemma}\label{L:compact}
Let $(X,d,m)$ be a metric measure space. For any compact set $K$ and any $\ve > 0$ there exists $n \in \N$ and $K_{i} \subset K$ compact for $i = 1, \dots, n$ such that
\[
diam(K_{i}) \leq \ve,  \qquad m\left( K \setminus \bigcup_{i =1}^{n} K_{i} \right) \leq \ve,
\]
and $K_{i}\cap K_{j} = \emptyset$ for $i\neq j$.
\end{lemma}

\begin{proof}
So let $\ve >0$ be given. Then consider the open covering of $K$ given by $\{ B_{\ve}(x)\}_{x\in K}$. By compactness, there exists  finitely many $\{x_{i}\}_{i\leq n}$ so that every point of $K$ is at distance less than $\ve$ for some $x_{i}$. Then consider the compact sets $H_{i} : = K \cap \overline B_{\ve}(x)$ for $i =1,\dots,n$.
Clearly the union of all $H_{i}$ covers $K$ and each of $H_{i}$ has diameter less than $\ve$. Taking differences we can pass to a family of Borel sets $\hat H_{i}$ so that 
\[ 
diam (\hat H_{i}) \leq \ve, \qquad \bigcup_{i =1}^{n} \hat H_{i} = K.
\]
with $\hat H_{i} \cap \hat H_{j} = \emptyset$ if $i \neq j$. Then by inner regularity with compact sets, choose for each $i \leq n$ a compact set $K_{i} \subset \hat H_{i}$ so that 
$m(\hat H_{i} \setminus K_{i}) \leq \ve/n$. The claim follows.
\end{proof}

\begin{lemma}\label{L:mcclocal}
Let $(X,d,m)$ be a non-branching metric measure space. Suppose that for each $K\subset X$ compact there exist $\delta,\ve >0$ and 
a measurable function $f : [0,\delta] \to (0,\infty)$ with $\limsup_{t\to 0}f(t) > 1/2$, so that 
\[
m(A_{t,x}) \geq f(t) m(A),\quad \forall t \in [0,\delta],
\]
for any $x \in K$ and $A\subset K$ compact with $diam (A) \leq \ve$. Then $(X,d,m)$ verifies Assumption \ref{A:mcc}. 
\end{lemma}

\begin{proof}
Consider $K \subset X$ compact set. Let $\delta, \ve > 0$ and the measurable map $f$ given by the hypothesis. Fix also $x \in K$.
Let $A \subset K$ be any compact set. Now for any $\eta < \ve$ consider the finite family of disjoint compact $\{ A_{i}\}_{i \leq n(\eta)}$ sets given by Lemma \ref{L:compact}.
Then since the space is non-branching and $diam(A_{i})\leq \ve$ it follows that 
\begin{align*}
m(A_{t,x}) = &~ \sum_{ i \leq n(\eta)} m((A_{i})_{t,x})  \crcr
\geq &~ f(t) \sum_{ i \leq n(\eta)} m(A_{i}) \crcr
\geq &~ f(t) m(A) -\eta f(t),
\end{align*}
for all $t \in [0,\delta]$. Since $\eta$ was any positive number less than $\ve$ and $\delta$ depends only on $K$ and $\ve$, the claim follows.
\end{proof}

It follows from Lemma \ref{L:mcclocal} that to verify Assumption \ref{A:mcc} it is sufficient to consider compact sets of small diameter. 
This already suggests that Assumption \ref{A:mcc} is stable under continuous changes of the measure as the one we proposed few lines above.
We now state and prove this stability property.

\begin{theorem}\label{T:newdensity}
Let $(X,d,m)$ be a non-branching metric measure space verifying Assumption \ref{A:mcc}. Consider a continuous function $g: X \to (0,\infty)$ and the measure $\tilde m : = g\cdot m$. 
Then $(X,d,\tilde m)$ is a non-branching metric measure space verifying Assumption \ref{A:mcc}.
\end{theorem}

\begin{proof} 
{\it Step 1.}
Note first, that by continuity of $g$, $\tilde m$ is finite over compact sets and therefore $(X,d,\tilde m)$ is a non-branching, metric measure space.
Let $K \subset X$ be any compact set and $\delta >0$ and $f$ measurable be given by Assumption \ref{A:mcc} for $(X,d,m)$. 
Note that $diam(K)$ is bounded, say by $M >0$.
Then for any $A \subset K$ compact, $x \in K$ and $t \in [0,\delta]$ the following chain of inequalities holds:
\begin{align*}
\tilde m (A_{t,x}) = &~ \int_{A_{t,x}} g(x) m(dx) \crcr
\geq &~ \inf \{ g(x) : x \in A_{t,x} \} \, m(A_{t,x}) \crcr
\geq &~ \inf \{ g(x) : x \in A_{t,x} \} \, f(t) m(A) \crcr
\geq &~ \frac{\inf \{ g(x) : x \in A_{t,x} \} }{\max \{ g(x) : x \in A \}}\, f(t) \int_{A}g(x)m(dx) \crcr
= &~\frac{\inf \{ g(x) : x \in A_{t,x} \} }{\max \{ g(x) : x \in A \}}\, f(t) \tilde m(A).
\end{align*}
Moreover from Lemma \ref{L:mcclocal} it follows that we can focus only on compact $A$ with arbitrarily small diameter. 

{\it Step 2.} Then we reason as follows: consider $\eta > 0$ so that 
\[
\left( 1 -\frac{\eta}{\alpha} \right) \limsup_{t \to 0} f(t) > \frac{1}{2},
\]
where $\alpha >0$ is so that $g(x) > \alpha$ for all $x \in K$.
Then since $g$ is uniformly continuous over $K$, there exists $\ve >0$ so that $|g(z) - g(w)| \leq \eta$ whenever $d(z,w) \leq 2 \ve$ for $z,w \in K$.

Let now $A \subset K$ be any compact set with $diam (A) \leq \ve$ and take $t \leq \min\{\delta,  \ve/M\}$. 
Then if $z \in A_{t,x}$ and $w \in A$, it follows that $d(z,w) \leq 2\ve$:
indeed there exists a geodesic $\gamma$ so that $\gamma_{0} \in A$, $\gamma_{1} = x$ and $\gamma_{t} = z$, then 
\[
d(z,w) \leq d(z,\gamma_{0}) + d(\gamma_{0},w) \leq t \cdot diam (K) + \ve \leq 2\ve.
\]

Then if $A \subset K$ compact set with $diam (A) \leq \ve$, $x \in K$ and $t \leq \min\{\delta,  \ve/M\}$ and
$x_{M} \in A$ so that $g(x_{M}) =  \max \{ g(x) : x \in A \}$, we have:
\[
1 - \frac{\inf \{ g(x) : x \in A_{t,x} \} }{\max \{ g(x) : x \in A \}} = \frac{\sup\{ g(x_{M}) - g(z) : z \in A_{t,x} \}}{g(x_{M})} \leq \frac{\eta}{\alpha},
\]
and therefore 
\[
\tilde m (A_{t,x}) \geq \left( 1 -\frac{\eta}{\alpha} \right) f(t) \tilde m(A).
\]
By the choice of $\eta$ we have proved Assumption \ref{A:mcc} for all compact sets with diameter smaller than $\ve$. Lemma \ref{L:mcclocal} gives the claim.
\end{proof}

Nevertheless if $(X,d,m)$ is also proper, Assumption \ref{A:mcc} carries some geometric property of the space.
\begin{proposition}
 Any proper, non-branching, metric measure space $(X,d,m)$ satisfying Assumption \ref{A:mcc} is locally doubling.
\end{proposition}
\begin{proof}
Take any ball $B_{2r}$ of radius $2r$. Fix $0<t\leq\delta$ and $n$ such that $(1-t)^n\leq 1/2$. 
Contracting $B_{2r}$ to its center yields $(B_{2r})_t=B_{(1-t)2r}.$ Contracting $B_{(1-t)2r}$ to 
its center yields $(B_{(1-t)2r})_t=B_{(1-t)^2 2r}$. 
Since $(X,d,m)$ is proper we can use Assumption \ref{A:mcc} and estimate
$$m(B_{(1-t)^2 2r})\geq f(t) m(B_{(1-t)2r}) \geq f(t)^2 m(B_{2r}).$$
Repeating this another $n-2$ times yields
$$m(B_r)\geq m(B_{(1-t)^n2r})\geq f(t)^n m(B_{2r}).$$
\end{proof}

\begin{remark}
Assume that $(X,d,m)$ is locally doubling and for any compact set $K \subset X$ there exists $0 < \delta \leq 1$ such that for any $t \leq \delta$ there exists a map 
$F_{t} : K \times K \to X$
such that 
\[
d(x,F_{t}(x,y)) = t d(x,y), \qquad 
\frac{1}{L(t)} \,d(x,z) \leq d(F_{t}(x,y), F_{t}(z,y)   ) \leq L(t)\, d(x,z)
\]
and $L(t) \to 1$ as $t$ goes to 0. Moreover, assume that $F_{t}$ varies continuously in time and for all compact sets $K$:
\begin{equation*}
\limsup_{t \to 0} \inf  \left\{ \frac{m(B_{r}(F_{t}(x,y)  ))}{ m(B_{r}(x)) } :  x,y \in K, r > 0 \right\} > \frac{1}{2}.
\end{equation*}
Then it is not hard to show using covering theorems that $(X,d,m)$ verifies Assumption \ref{A:mcc}.

This says that a certain type of Ahlfors regularity together with a bi-Lipschitz selection of t-intermediate points implies Assumption \ref{A:mcc}.
\end{remark}

\section{Evolution estimates}\label{S:evolution}
Following Section \ref{S:notation}, we fix once for all $(X,d,m)$ a proper, non-branching, metric measure space verifying 
Assumption \ref{A:mcc}, two probability measures $\mu_{0}, \mu_{1}$ with $\mu_0\ll m$ and 
$h : [0,\infty) \to [0, \infty)$ strictly convex and non-decreasing. 

Since we have to prove a local property, we can assume that 
$\supp(\mu_{0}),\supp(\mu_{1}) \subset K$ with $K$ compact.  Then by standard results in optimal transportation, 
there exists a couple of Kantorovich potentials $(\f,\f^{c})$ such that if 
\[
\Gamma : = \{ (x,y)\in X \times X : \f(x) +\f^{c}(y) = c(x,y)\},
\]
then the transport plan $\pi$ is optimal iff $\pi(\Gamma)=1$ (e.g. see Theorem 5.10 in \cite{villa:Oldnew}). Note also that the set $\Gamma$ is $c$-cyclically monotone.
So also $K,\f,\f^{c}$ and $\Gamma$ are fixed.

We start by proving the standard property of geodesics belonging to the support of the optimal dynamical transference plan $\pi$: they cannot meet at the same time $t$ if $t\neq 0, 1$. For existence results and details on dynamical transference plans we refer to \cite{villa:Oldnew} Chapter 7.

\begin{lemma}\label{L:noncrossing}
Let $(x_{0},y_{0}), (x_{1},y_{1}) \in \Gamma$ be two distinct points.
Then for any $t \in (0,1)$,  
\[
d(x_{0}(t),x_{1}(t))>0,
\]
where $x_{i}(t)$ is any $t$-intermediate point between $x_{i}$ and $y_{i}$, for $i =0,1$.
\end{lemma}

\begin{proof} 

Assume by contradiction the existence of $x_{0}(t) =x_{1}(t) \in X$, 
$t$-intermediate points of $(x_{0},y_{0})$ and $(x_{1},y_{1})$, i.e. 
\[
d(x_{0},x_{0}(t)) = t d(x_{0},y_{0}), \qquad d(x_{0}(t),y_{0}) = (1-t) d(x_{0},y_{0}),
\]
and 
\[
d(x_{1},x_{1}(t)) = t d(x_{1},y_{1}), \qquad d(x_{1}(t),y_{1}) = (1-t) d(x_{1},y_{1}).
\]
{\it Case 1:} $d(x_{0},y_{0}) \neq d(x_{1},y_{1})$. Then
\begin{align*}
h(d(x_{0},y_{1})) + h(d(x_{1}, y_{0})) \leq &~ 
h\big(d(x_{0}, x_{0}(t)) + d(x_{1}(t),y_{1})\big) + 
h\big(d(x_{1}, x_{1}(t)) + d(x_{0}(t),y_{0})\big) \crcr
 < &~ t h(d(x_{0}, y_{0}))+ (1-t) h(d(x_{1}, y_{1}))\crcr 
&~  +  t h(d(x_{1}, y_{1})) + (1-t) h(d(x_{0}, y_{0}))  \crcr
= &~    h(d(x_{0}, y_{0}))+ h(d(x_{1}, y_{1})).
\end{align*}
Where between the first and the second line  we have used the strict convexity of $h$.
From $c$-cyclical monotonicity we have a contradiction.

{\it Case 2.}  $d(x_{0},y_{0}) = d(x_{1},y_{1})$. Let $\gamma^{0}, \gamma^{1} \in \G(X)$ be such that 
\[
\gamma^{0}_{0} = x_{0}, \quad \gamma^{0}_{t} = x_{0}(t), \quad \gamma^{0}_{1} = y_{0}, \qquad \qquad
\gamma^{1}_{0} = x_{1}, \quad \gamma^{1}_{t} = x_{1}(t), \quad \gamma^{1}_{1} = y_{1},
\] 
and define the curve $\gamma : [0,1] \to X$ by
\[
\gamma_{t} : = 
\begin{cases}
\gamma^{0}_{s}, & s \in [0,t] \crcr
\gamma^{1}_{s}, & s \in [t,1].
\end{cases} 
\]
Then $\gamma$ is a geodesic different from $\gamma^{0}$ but coinciding with it on the non trivial interval $[0,t]$. Since this is a contradiction with the 
non-branching assumption, the claim is proved.
\end{proof}

\begin{remark}
In the framework of metric measure spaces enjoying synthetic Ricci curvature bounds, like $\mathsf{CD}(K,N)$, see \cite{sturm:MGH2} for its definition,
it has recently been shown by Rajala that assuming the convexity of the entropy along all $L^{2}$-Wasserstein geodesics implies that any optimal transport plan is concentrated 
on a family of non-branching geodesics, even if the space is not assumed to be non-branching. 
Unfortunately, in our framework such a technique cannot be used, at least for now. Indeed while all the curvature information are stated in terms of $L^{2}$-Wasserstein geodesics,
here we would need a non-branching property of the geodesics of the space $X$ with final and initial points forming a $c$-cyclically monotone set. The latter property cannot be deduced straightforwardly by $d^{2}$-monotonicity. 
For the moment the only result going in this direction is for $h = id$ and it is proven in \cite{cava:RCDmonge}.
\end{remark}

For any compact set $\Lambda \subset X \times X$ we can now consider the associated evolution map. 
For every $t \in [0,1]$ and every $A \subset X$ compact set
\[
A_{t,\Lambda} : = e_{t} \left( (e_{0},e_{1})^{-1}\left( (A \times X) \cap \Lambda\right) \right).
\]
It is easily seen that $A_{t,\Lambda}$ is a closed and bounded set. Hence since $(X,d,m)$ is proper we also obtain compactness of $A_{t,\Lambda}$.
Moreover we will use the following notation: to any $\Lambda \subset X\times X$ we associate 
the following set:
\begin{equation}\label{E:square}
\hat \Lambda : = \left(P_{1}(\Lambda) \times P_{2}(\Lambda) \right)\cap \Gamma.
\end{equation}

We are now ready to prove the main consequence of Assumption \ref{A:mcc}.

\begin{proposition}\label{P:evolution}
For any $\Lambda \subset \Gamma$ compact the following inequality holds: 
\begin{equation}\label{E:evolution}
m(A_{t,\hat \Lambda}) \geq f(t) m(A), \qquad \, \, t \in [0,\delta],
\end{equation}
for any $A \subset P_{1}(\Lambda)$.
\end{proposition}

\begin{proof}
{\it Step 1.} Let $\{ y_{i} \}_{i \in \enne} \subset P_{2}(\Lambda)$ be a dense set in $P_{2}(\Lambda)$.

Consider the following family of sets: for $n \in \enne$ and $i \leq n$
\[
E_{n}(i) : = \{ x \in P_{1}(\Lambda) : c(x,y_{i}) - \f^{c}(y_{i}) \leq c(x,y_{j}) - \f^{c}(y_{j}), j=1,\cdots,n \}.
\]
If we now consider  
\[
\Lambda_{n} : = \bigcup_{i =1 }^{n}  E_{n}(i)\times \{y_{i}\},  
\]
it is straightforward to check that $P_{1}(\Lambda_{n}) =  P_{1}(\Lambda)$ and $\Lambda_{n}$ is $c$-cyclically monotone. Indeed, for any 
$(x_{1},y_{1}), \dots, (x_{m},y_{m}) \in \Lambda_{n}$, by definition it holds that
\[
c(x_{i},y_{i}) - \f^{c}(y_{i}) \leq c(x_{i},y_{i+1}) - \f^{c}(y_{i+1}), \qquad i =1,\dots, m.
\]
Taking the sum over $i$, the property follows. 

By Assumption \ref{A:mcc} there exists $f : [0,1] \to \erre$ measurable with $f(0)> 1/2$, 
independent of the sequence $\{y_{i}\}_{i\in \enne}$ and of $n$, such that  for any $A \subset P_{1}(\Lambda)$ compact it holds that
\[
m \left( \left(A\cap E_{n}(i)\right)_{t,y_{i}} \right) \geq f(t) m(A\cap E_{n}(i)), \qquad \forall t \in [0,\delta],
\] 
where $\left(A\cap E_{n}(i)\right)_{t,y_{i}} = (A\cap E_{n}(i))_{t, E_{n}(i) \times \{y_{i} \}}$. 
Note that since $A = \cup_{i\leq n} A \cap E_{n}(i)$ it follows that
\begin{align*}
A_{t,\Lambda_{n}} = &~ e_{t} \left( (e_{0},e_{1})^{-1}((A \times X) \cap \Lambda_{n}) \right) \crcr
= & ~\bigcup_{i\leq n} e_{t}\left( (e_{0},e_{1})^{-1}(((A\cap E_{n}(i)) \times X) \cap \Lambda_{n}) \right) \crcr
= &~\bigcup_{i\leq n} (A\cap E_{n}(i))_{t,\Lambda_{n}} \crcr
\supset &~\bigcup_{i\leq n} (A\cap E_{n}(i))_{t,E_{n}(i)\times\{y_{i}\}}.
\end{align*}
Moreover,
Lemma \ref{L:noncrossing} implies
\[
\left(A\cap E_{n}(i)\right)_{t,y_{i}} \cap  \left(A\cap E_{n}(j)\right)_{t,y_{j}} = \emptyset, \qquad i\neq j,
\]
for all $t \in (0,1)$.

Then it holds for all $t \in [0,\delta]$:
\begin{align}
m(A_{t,\Lambda_{n}}) \geq &~ m \left( \bigcup_{i=1}^{n}\left( A \cap E_{n}(i)  \right)_{t,E_{n}(i)\times\{y_{i}\}}\,   \right)  = 
\sum_{i=1}^{n} m \left( \left( A \cap E_{n}(i)  \right)_{t,y_{i}}   \right)  \crcr
\geq &~ f(t)\sum_{i=1}^{n} m \left(  A \cap E_{n}(i)  \right)   \crcr
\geq &~ f(t) m \left( \bigcup_{i=1}^{n}  A \cap E_{n}(i)   \right) \crcr
= &~ f(t)m (A ).
\end{align}

{\it Step 2.} Note that for every $n \in \enne$, $\Lambda_{n} \subset \supp(\mu_{0}) \times \supp(\mu_{1})$ and the latter, by assumption, is a subset of $K \times K$. Since the space of closed subsets of $K \times K$ endowed with the Hausdorff metric $(\mathcal{C}(K\times K), d_{\mathcal{H}})$ is a compact space, 
there exists a subsequence $\{\Lambda_{n_{k}} \}_{k \in \enne}$ and $\Theta \subset K\times K$ compact such that 
\[
\lim_{k\to \infty} d_{\mathcal{H}}(\Lambda_{n_{k}}, \Theta) = 0.
\]
Since the sequence $\{y_{i}\}_{i\in \enne}$ is dense in $P_{2}(\Lambda)$ and $\Lambda \subset \Gamma$ is compact, by definition of $E_{n}(i)$,
necessarily for every $(x,y) \in \Theta$ it holds
\[
\f(x)+ \f^{c}(y) = c(x,y), \quad x \in P_{1}(\Lambda), \quad y \in P_{2}(\Lambda). 
\]
Hence $\Theta \subset \left( P_{1}(\Lambda) \times P_{2}(\Lambda) \right) \cap \Gamma = \hat \Lambda$. 
To conclude the proof we have to observe 

\[
m(A_{t,\Theta}) \geq \limsup_{k\to \infty} m(A_{t,\Lambda_{n_{k}}})\, .
\]
Indeed, since $A_{t,\Theta}$ is a compact set, it follows that if $A_{t,\Theta}^{\ve} = \{ x \in X : d(x, A_{t,\Theta})\leq\ve\}$, then 
for $k$ sufficiently big $A_{t,\Lambda_{n_{k}}} \subset A_{t,\Theta}^{\ve}$ and $m(A_{t,\Theta}^{\ve})$ converges to $m(A_{t,\Theta})$.

Then
\[
m(A_{t,\hat\Lambda}) \geq \limsup_{k\to \infty} m(A_{t,\Lambda_{n_{k}}}) \geq f(t) m(A),
\]
and the claim follows.
\end{proof}

\section{Existence of optimal maps}\label{S:final}

In this section we show that branching at starting points does not happen almost surely. Recall that 
\[
\Gamma = \{ (x,y)\in X \times X : \f(x) +\f^{c}(y) = c(x,y)\}
\]
and any optimal transport plan is concentrated on $\Gamma.$

\begin{lemma}\label{L:disjointevo}
Let $\Lambda_{1},\Lambda_{2} \subset \Gamma$ be compact sets such that
\begin{itemize}
\item[$i)$] $P_{1}(\Lambda_{1}) = P_{1}(\Lambda_{2})$;
\item[$ii)$] $P_{2}(\Lambda_{1}) \cap P_{2}(\Lambda_{2}) = \emptyset$.
\end{itemize}
Then $m(P_{1}(\Lambda_{1})) = m(P_{1}(\Lambda_{2})) = 0$.
\end{lemma}
\begin{proof}
Note that since $P_{2}(\Lambda_{1}) \cap P_{2}(\Lambda_{2}) = \emptyset$, necessarily $\hat \Lambda_{1} \cap \hat \Lambda_{2} = \emptyset$, 
where $\hat \Lambda_{i}$ are defined by \eqref{E:square}, for $i = 1,2$.
Hence from Lemma \ref{L:noncrossing}, for every $A \subset  P_{1}(\Lambda_{1}) = P_{1}(\Lambda_{2})$
\[
A_{t,\hat \Lambda_{1}} \cap A_{t,\hat \Lambda_{2}} = \emptyset,
\]
for every $t \in (0,1)$. Then let $A : = P_{1}(\Lambda_{1}) = P_{1}(\Lambda_{2})$ and recall that as $t \to 0$ the sets $A_{t,\Lambda_{1}}$ and
$A_{t,\Lambda_{2}}$ both converge in Hausdorff topology to $A$. Put $A^\ve=\{x:d(x,A)\leq \ve\}.$ 
Then it follows from Proposition \ref{P:evolution} that
\begin{align*}
m(A) = &~ \limsup_{\ve \to 0} m(A^{\ve}) \geq \limsup_{t\to 0} m(A_{t,\Lambda_{1}} \cup A_{t,\Lambda_{2}}) \crcr
= & ~\limsup_{t\to 0} \big(m(A_{t,\Lambda_{1}} ) + m( A_{t,\Lambda_{2}})\big) \crcr
\geq &~  m(A) \limsup_{t \to 0} 2 f(t) = \alpha \cdot m(A), 
\end{align*}
with $\alpha > 1$. Hence, necessarily $m(P_{1}(\Lambda_{1})) = m(P_{1}(\Lambda_{2})) = m(A) = 0$, and the claim follows.
\end{proof}

We will use the following notation: $\Gamma(x): = (\{x \} \times X) \cap \Gamma$ and given a set $\Theta \subset X \times X$ we say that $T$ is a selection of $\Theta$ if $T: P_{1}(\Theta) \to X$ is $m$-measurable and $\gr(T) \subset \Theta$.

\begin{proposition}\label{P:existence}
Consider the sets
\[
E: = \{x \in P_{1}(\Gamma ) :\Gamma(x)\ is\, not\, a\, singleton\}, \qquad \Gamma_{E} : = \Gamma \cap (E \times X).
\]
Then for any selection $T$ of $\Gamma_{E}$ and every $\pi \in \Pi(\mu_{0},\mu_{1})$ with $\pi(\Gamma ) = 1$ it holds
\[
\pi (\Gamma_{E} \setminus \gr(T)) = 0.
\]
\end{proposition}

\begin{proof}

{\it Step 1.}
Suppose by contradiction the existence of $\pi \in \Pi(\mu_{0},\mu_{1})$ with $\pi(\Gamma) = 1$ and of a selection $T$ of $\Gamma_{E}$ such that 
\[
\pi(\Gamma_{E} \setminus \gr(T)) = \beta >0.
\]
By inner regularity, to prove the complete statement it is enough to prove it under the additional assumptions
that $E$ is compact and $T$ is continuous.

Note that 
\[
\Gamma_{E} \setminus \gr(T) = \bigcup_{n =1}^{\infty} \{ (x,y) \in \Gamma_{E} : d(y,T(x))\geq1/n \}\, .
\]
Hence, there exists $n \in \enne$ such that 
\[
\pi\left(\{ (x,y) \in \Gamma_{E} : d(y,T(x))\geq1/n \}\right) \geq \beta' >0.
\]
Put $\Lambda:= \{ (x,y) \in \Gamma_{E} : d(y,T(x))\geq1/n \}$. Note that $m(P_{1}(\Lambda)) > 0$.

{\it Step 2.} From the continuity of $T$ it follows the existence of $\eta >0$ so that if $d(x,z) \leq \eta$ then 
$d(T(x),T(z)) \leq 1/2n$. 
Clearly we can take $x \in P_{1}(\Lambda)$ so that 
\[
m(P_{1}(\Lambda) \cap \bar B_{\eta}(x))>0.
\]
where $\bar B_\eta(x)$ denotes the closed ball of radius $\eta$ around $x$. So consider the two sets
\[
\Xi_{1} : = \gr(T) \cap \left(\left(\bar B_{\eta}(x) \cap P_{1}(\Lambda)\right) \times X \right), \quad \Xi_{2} : = 
\left(\bar B_{\eta}(x) \times X\right) \cap \Lambda. 
\]
By construction $\Xi_{1},\Xi_{2}\subset \Gamma$ and 
\[
P_{1}(\Xi_{1}) = P_{1}(\Xi_{2}) = P_{1}(\Lambda) \cap \bar B_{\eta}(x),
\]
therefore $m(P_{1}(\Xi_{1}))>0$. 

Moreover for any $y \in P_{2}(\Xi_{2})$ there exists $w \in \bar B_{\eta}(x)$ so that 
\[
d(y,T(w)) \geq \frac{1}{n}. 
\]
Hence for any $z \in \bar B_{\eta}(x)$ it holds
\[
d(y,T(z)) \geq d(y,T(w)) - d(T(w),T(z)) \geq \frac{1}{n} - \frac{1}{2n} = \frac{1}{n}.
\]
Hence 
\[
P_{2}(\Xi_{1}) \cap  P_{2}(\Xi_{2}) = \emptyset.
\]
Since this is in contradiction with Lemma \ref{L:disjointevo}, the claim is proved.
\end{proof}

We can now state the main result of this paper whose 
 proof now follows as a straightforward corollary of what we proved so far.

\begin{theorem}\label{T:main}
Let $(X,d,m)$ be a non-branching metric measure space verifying Assumption \ref{A:mcc}. Let $\mu_0$ and $\mu_1$ be two probability measures over $X$ with finite $c$-transport distance.
If $\mu_0\ll m$ and $h$ is strictly-convex and non-decreasing,
for any $\pi \in \Pi(\mu_{0},\mu_{1})$ such that $\pi(\Gamma) = 1$ there exists an $m$-measurable map $T:X \to X$ such that 
\[
\pi(\gr(T)) =1.
\]
\end{theorem}
\begin{proof}
Let $\pi \in \Pi(\mu_{0},\mu_{1})$ be any transference plan so that $\pi(\Gamma)=1$.
As for Proposition \ref{P:existence}, consider the sets
\[
E: = \{x \in P_{1}(\Gamma ) :\Gamma(x)\ is\, not\, a\, singleton\}, \qquad \Gamma_{E} : = \Gamma \cap (E \times X).
\]
Since
\[
\Gamma_{E} = P_{12}\left( \{ (x,y,z,w) \in \Gamma \times \Gamma : d(x,z) =0, d(y,w)>0\}   \right),
\]
the set $\Gamma_{E}$ is an analytic set. For the definition of analytic set, see Chapter 4 of \cite{Sri:courseborel}. 
We can then use the Von Neumann Selection Theorem for analytic sets, see Theorem 5.5.2 of \cite{Sri:courseborel}, to obtain a map $T : E \to X$, $\mathcal{A}$-measurable, where $\mathcal{A}$ is the 
$\sigma$-algebra generated by analytic sets, so that $(x,T(x)) \in \Gamma_{E}$.

Then Proposition \ref{P:existence} implies that 
\[
\pi\llcorner_{\Gamma_{E}} = (Id,T)_{\sharp} \mu_{0}\llcorner_{E}.
\]
Since on $\Gamma\setminus \Gamma_{E}$ $\pi$ is already supported on a graph, the claim follows.
\end{proof}

This directly implies

\begin{corollary}
Under the assumptions of Theorem \ref{T:main}, there is a unique optimal transport map.
\end{corollary}
\begin{proof}
The last theorem shows that every optimal coupling is induced by a transport map. As the set of all optimal couplings is convex this directly implies the uniqueness.
\end{proof}

\end{document}